\documentclass[a4paper,10pt]{article}
\usepackage{amsmath}
\usepackage{amssymb}

\usepackage{amscd}
\usepackage{amsthm}
\usepackage{url}
\usepackage{graphicx}

\usepackage{xcolor}

\newtheorem{thm}{Theorem}
\newtheorem{cor}[thm]{Corollary}

\newtheorem{lem}[thm]{Lemma}

\newtheorem{defin}[thm]{Definition}
\newtheorem{prop}[thm]{Proposition}

\title{On the Topology of the Inverse Limit of a Branched Covering over a
Riemann Surface.}
\author{Carlos Cabrera\footnote{UNAM, Mexico.}, Chokri Cherif\footnote{BMCC/CUNY, USA.} and Avraham Goldstein\footnote{CUNY, USA. Email: avraham.goldstein.nyc@gmail.com.}}

\begin{document}
\maketitle

\begin{abstract}
We introduce the Plaque Topology on the inverse limit of a branched covering self-map
of a Riemann surface of a finite degree greater than one. We present the notions of regular and irregular points in the setting of this Plaque Inverse Limit and study its local topological properties at the irregular points. We construct a certain
Boolean algebra and a certain sigma-lattice, derived from it, and use them to compute local topological invariants of the Plaque Inverse Limit. Finally, we obtain several results interrelating the dynamics of the forward iterations of the self-map and the topology of the Plaque Inverse Limit.
\end{abstract}

{\bf Keywords:} Inverse limit, Riemann surface lamination, holomorphic dynamics, branched coverings, local topology, irregular points, Boolean algebra

\section{Introduction.}

Topological inverse limits of dynamical systems were constructed and their topological
properties were studied in literature since the late 1920s. The most famous
classical examples of such inverse limits are the solenoids, which are defined as the
inverse limits of the iterates of the $d$-fold covering self-mapping $f(z)=z^d$ (where $d>1$) of the unit circle $S^1$.
The inverse limit of these iterates, for a fixed integer $d$,
is called the $d$-adic solenoid. It is a compact, metrizable topological space that is connected,
but neither locally connected nor path connected. Solenoids were first introduced by L. Vietoris in 1927 for $d=2$ (see \cite{Vietoris}) and later in 1930 by van Dantzig for an arbitrary $d$ (see \cite{Danzig}).
\\ \\
D. Sullivan in \cite{Sullivan} introduced Riemann surface laminations, which arise when taking
inverse limits in dynamics. A Riemann surface lamination is locally the product of a complex
disk and a Cantor set. In particular, D. Sullivan associates such lamination to any smooth,
expanding self-mapping of the circle $S^1$, with the maps $f(z)=z^d$ being examples of such mappings.
\\ \\
M. Lyubich and Y. Minsky took it one dimension higher. In \cite{LM} they consider dynamics of
rational self-mappings of the Riemann sphere and introduce three-dimensional laminations associated
with these dynamics. Thus, the theory of Riemann surface laminations associated with holomorphic
dynamics was first founded and formalized by M. Lyubich and Y. Minsky in \cite{LM} and, in parallel, by M. Su in \cite{Su}. The notions
of regular points and, by complement, irregular points were introduced in \cite{LM}. First, M. Lyubich and Y. Minsky consider the standard (Tychonoff) inverse limit of the iterations of a
rational map applied to the Riemann sphere, regarding them just as iterations of a continuous branched
covering map applied to a Hausdorff topological space. They call this inverse limit the natural extension of
the original dynamical system. Next, they define a point of this natural extension to be regular if the
pull-back of some open neighborhood of its first coordinate along that point will eventually become univalent.
The set of all regular points of the natural extension, which is the natural extension with
all the irregular points removed, is called the regular set. The Riemann surface lamination, which in the Lyubich-Minsky theory is associated with a holomorphic dynamical system, in many cases, is just the regular set. In general, certain modifications are performed to the regular set, in order to satisfy the requirement, that the conformal structure on the leaves of the Riemann surface lamination is continuous along the fiber of the lamination. For the details of Lyubich-Minsky's definition and construction of the Riemann surface lamination, which are somewhat elaborate, we refer to \cite{LM}.
\\ \\
In this paper we consider the following two questions:
\begin{itemize}
 \item How can the irregular points be distinguished and characterized?
 \item What is the relationship between the dynamics of the system and the characterization of the irregular points?
\end{itemize}

For our purposes it is more natural to equip the inverse limit with
the box topology, which is more refined than the Tychonoff topology.
\\ \\
An inverse dynamical system is a sequence:
$$S_1\;^{\underleftarrow{\;\;f_1\;\;}} S_2\;
^{\underleftarrow{\;\;f_2\;\;}} S_3\;...$$ of Riemann surfaces $S_i$
and holomorphic branched coverings $f_i:S_{i+1}\rightarrow S_i$ where
all $S_{i}$ are equal to a given Riemann surface $S_0$ and all $f_{i}$ are
equal to a given holomorphic branched covering map $f:S_0\rightarrow S_0$ of
degree $d$. In this work we assume that $1<d<\infty$ and $S_0$ is either the
unit disk, the complex plane or the Riemann sphere. We define the Plaque
Inverse Limit [P.I.L.] $S_{\infty}$ of that inverse dynamical system to be the following topological space: its underlying set is the set of all
the sequences $x=(x_1\in S_1,x_2\in S_2,...)$ of points, such that
$f_{i}(x_{i+1})=x_i$ for $i=1,2,...$; its topology is the family of
all the sequences $U=(U_1\subset S_1,U_2\subset S_2,...)$ of open sets, such that
$f_{i}(U_{i+1})=U_i$ for $i=1,2,...$. Let $p_i$, for all $i$, be the map from
$S_{\infty}$ onto $S_i$, which takes $(x_1,x_2,...)\in S_{\infty}$ to $x_i\in
S_i$. The maps $p_i$ are continuous and satisfy $f_i\circ p_{i+1}=p_i$. They are called the projection maps from $S_{\infty}$ onto $S_i$.
\\ \\
The standard Inverse Limit $\bar{S}_{\infty}$ of
$f$, as a set, is defined exactly as P.I.L., but is equipped with the
Tychonoff topology, in which the open sets are
all sequences $U=(U_1\subset S_1,U_2\subset S_2,...)$ of open sets, where
$f_{i}(U_{i+1})=U_i$ for $i=1,2,...$, such that there exists some number $t$, so that $f^{-1}_{i}(U_{i})=U_{i+1}$ for all $i>t$.
Thus, P.I.L. has more open sets than the standard Inverse Limit.
So, the identity map from the P.I.L. onto the Inverse Limit is
continuous.
\\ \\
Obviously, the above mentioned projection maps $p_i$ are also continuous, as maps from
$\bar{S}_\infty$ onto $S_i$. Actually, the categorical definition of the
topology of the inverse limit of an inverse system is precisely the ``minimal"
topology, which makes these projection maps $p_i$ continuous. Minimality, in this context, means
that any other topological space, together with maps from it into the inverse
system, which commute with the maps of the inverse system, can be mapped into the inverse
system through the inverse limit, and this can be done in a unique way. It is important to notice, that the map $f$
induces an automorphism $f$ of the P.I.L. and an automorphism $\bar{f}$ of the
inverse limit.
\\ \\
A local base for the topology of $S_{\infty}$ at a point $x$ consists
of all open sets $U$, containing $x$, such that each $U_i$ is conformally equivalent
to the unit disk in the complex plane and $f_i$, restricted to
$U_{i+1}$, is conformally equivalent to some self-map $z^t$ of the unit disk of a degree $t$, where $1\le t\le d$. Such open sets $U$ are called plaques. When we speak of a
neighborhood of a point in $S_{\infty}$, we always assume it to be a plaque. Similarly, when we speak of a neighborhood of a point in a Riemann surface, we, unless otherwise stated, assume it to be simply connected.
A point $x\in S_{\infty}$ is called \textit{regular} if, for some
neighborhood $U$ of $x$, there exists $n$, such that $U_{n+i+1}$ contains no
critical points of $f_{n+i}$ for all $i=0,1,2,...$. Thus,
$f_{n+i}:U_{n+i+1}\rightarrow U_{n+i}$ is a conformal equivalence.
Otherwise, the point $x\in S_{\infty}$ is called \textit{irregular}.
\\ \\
The set $\Delta$ of all the regular points of $S_{\infty}$ is open and
each of its path-connected components has a Riemann surface
structure. These Riemann surfaces were studied and fully classified in
\cite{LM}. In this work we:
\begin{itemize}
\item show that $\Delta$ is not empty;
\item show that a point $x\in S_{\infty}$ is irregular if and only if
there exists a neighborhood $U$ of
$x$ such that, for any neighborhood $V$ of $x$, whose closure $\overline{V}$ is contained inside $U$,
the open set $V-\{x\}$ is an uncountable
union of pairwise disjoint path-connected components. Thus, at an irregular
point $x\in S_{\infty}$ the P.I.L. is not even a topological manifold;
\item develop and construct a $\sigma$-algebraic machinery allowing us to obtain
and compute certain local invariants of P.I.L.;
\item use these invariants to characterize irregular points of P.I.L.; and
\item relate these invariants to the
properties
of the original dynamical system.
\end{itemize}
We would like to thank the anonymous Referee for a great contribution to this work.
\section{Constructions and definitions.}
For a dynamical system $f:S_0\rightarrow S_0$, where $S_0$ is either
the unit disk, the complex plane, or the Riemann sphere, and $f$ is a
branched covering of degree $1<d<\infty$, we define an inverse
dynamical system $S$ as:
$$S_0\;^{\underleftarrow{\;\;f\;\;}} \; S_1\;^{\underleftarrow{\;\;f_1\;\;}}
S_2\;
^{\underleftarrow{\;\;f_2\;\;}}\;...,$$ where $S_0=S_1=S_2=...$
and $f=f_1=f_2=...$. We speak of $f$ as a map from $S_i$ to $S_{i-1}$ for
$i>0$.
\\ \\
We denote the critical points of $f$ by $c_1,...,c_k$.
\\ \\
Since $1\le k\le \chi(S_0)\cdot(d-1)$, where $\chi(S_0)$ is the Euler
characteristic of $S_0$, we get $1\le k\le 2(d-1)$ for the Riemann sphere and
$1\le k\le (d-1)$ for the complex plane and the unit disk.
\\ \\
The following three definitions and Lemma \ref{regularexist}, after them, apply also in the
more general case of an inverse system, in which one does not require all the Riemann surfaces $S_i$
and all the holomorphic branched coverings $f_i$ to be the same:

\begin{defin} The Plaque Inverse Limit [P.I.L.] $S_{\infty}$ of an inverse system
is the set of all sequences of points
$(x_1\in S_1,x_2\in S_2,...)$, such that $f_{i}(x_{i+1})=x_i$,
 equipped with the topology, generated by all the sequences of
open sets\\ $U=(U_1\subset S_1,U_2\subset S_2,...)$, such that
$f_{i}(U_{i+1})=U_i$.
\end{defin}
The P.I.L. $S_{\infty}$ is equipped with
continuous projection maps $p_i:S_\infty\rightarrow
S_i$, defined by $p_i(x_1,x_2,...)=x_i$. We have $f_i\circ
p_{i+1}=p_i$. These maps constitute part of the structure of $S_{\infty}$. Notice that $S_{\infty}$, as a topological space, is regular and first-countable. Abusing the notation, we will also call the underlying topological space $S_{\infty}$.

\begin{defin} An open set $U\subset S_{\infty}$ is called a plaque if each
$U_i\subset S_i$ is conformally equivalent
to the unit disk in $\mathbb{C}$ and each $f_i$, restricted to
$U_{i+1}$, is conformally equivalent to a self-map $z^t$ of the unit disk of a degree $t\leq deg(f_i)$.\end{defin}

All the plaques, containing a point $x\in S_{\infty}$, constitute
a local base for the topology of $S_{\infty}$ at $x$. In this work, whenever we consider an open
neighborhood of a point in a Riemann surface, we assume it to be simply connected. Similarly, in this work open
neighborhoods of points in $S_{\infty}$ are assumed to be plaques.

\begin{defin} A point $x\in S_{\infty}$ is called regular if there exists an
open neighborhood $U$ of $x$ such that for some positive integer $n$,
$f_{i}:U_{i+1}\rightarrow U_{i}$ are
bijections for all $i\geq n$. Otherwise, $x$ is called irregular.
\end{defin}

The set of all the regular points of $S_{\infty}$ is denoted by $\Delta$. In the following
lemma we require all $S_i$ to be a fixed Riemann surface $-$ either the Riemann sphere, the complex plane, or the unit disk.

\begin{lem} \label{regularexist} If, for all $i=1,2,...$, $deg(f_{i+1})\le 2*deg(f_i)$, then the set $\Delta$ is not empty. \end{lem}

\begin{proof} For all $i=1,2,...$ let $q_i=10deg(f_i)$. Take any $q_1$ nonempty, pairwise disjoint, simply connected open
sets $U_1(1),...,U_1(q_1)$ in $S_1$, which do not
contain any images of the critical points of $f_1$. Then, for each
$U_1(i)$, there are $deg(f_1)$ pairwise disjoint, simply connected, open sets in
$S_2$ which map onto $U_1(i)$ by $f_1$. Amongst these
$deg(f_1)\cdot q_1=10(deg(f_1))^2$ open pre-image sets in $S_2$, at least
$10(deg(f_1))^2-2(deg(f_2)-1)$ do not contain any
images of the critical points of $f_2$. Since $deg(f_1)\ge 2$, we obtain
$10(deg(f_1))^2-2(deg(f_2)-1)\ge 10(deg(f_1))^2-4deg(f_1)+2\ge 20deg(f_1)
\ge10deg(f_2)=q_2$. So, amongst these pre-images we can select $q_2$ open,
simply connected, pre-images $U_2(1),...,U_2(q_2)$, which do not contain any images of
the critical points of $f_2$. Repeating this process
an infinite number of times produces at least one non-empty, open, simply connected set
in $S_\infty$, which contains no critical points of any $f_i$.
\end{proof}
Now we introduce certain algebraic structures, which are used to compute
local invariants of the Plaque Inverse Limits.
The set of all binary sequences, equipped with the below listed structures, is a Boolean algebra:
\begin{itemize}
\item the operations $\vee$ and $\wedge$, defined by performing the binary operations \textbf{or} and
\textbf{and}, respectively, in each coordinate of the sequences;
\item a partial order $\le$ on the sequences, defined by: $b\le a$ if and
only if $a\vee
b=a$. This is equivalent to defining $b\le a$ if and only if every entry of
$b$ is less than or equal to the corresponding entry of $a$;
\item with the order above, there is a minimal element $(0,0,0,...)$ and a
maximal element $(1,1,1,...)$;
\item the negation operation $\neg$, which interchanges 0 and 1 in every coordinate of the binary sequence.
\end{itemize}
Two binary sequences are called \textit{almost equal} if they
differ only in a finite number of places. This ``almost equality" is an
equivalence
relation. Additionally, it respects the $\vee$, $\wedge$, $\neg$
operations, the partial order $\le$, and the minimal and maximal elements.

\begin{defin} The set $I$ is the Boolean algebra of all classes of almost equal
binary sequences, equipped with the $\vee$ and $\wedge$ operations which are
defined as follows:
$$[a_1,a_2,...]\vee [b_1,b_2,...]=[a_1\vee b_1,a_2\vee b_2,...]$$ and
$$[a_1,a_2,...]\wedge [b_1,b_2,...]=[a_1\wedge b_1,a_2\wedge b_2,...].$$
Its minimal element is $\textbf{0}=[0,0,...]$ and its maximal element is
$\textbf{1}=[1,1,...]$. Its negation is $\neg[a_1,a_2,...]=[\neg a_1,\neg
a_2,...]$.
\end{defin}
\begin{defin}
For every $a\in I$, we define $\alpha(a)\subset I$ as the set of all
$b\in I$ such that $b\le a$.
\end{defin}
Note that $\alpha(a)\cup \alpha(b)\subset \alpha(a\vee b)$,
$\alpha(a\wedge b)=\alpha(a)\cap \alpha(b)$,
$\alpha(\textbf{0})=\{\textbf{0}\}$,
and $\alpha(\textbf{1})=I$.

\begin{defin} \label{categ}
The $\sigma$-lattice $A$, spanned by all $\alpha(a)$, where
$a\in I$, with the operations $\cup$ and $\cap$, the minimal
element $\{\textbf{0}\}$, and the maximal element $I$, is called the
signature $\sigma$-lattice. The elements of $A$ are called signatures.
\end{defin}

It is clear that $\subset$ defines a partial order on $A$. This
partial order is consistent with the partial order $\leq$ of $I$ under the map
$\alpha$, since if $b\le a$ then $b\in \alpha(a)$. Therefore, $\alpha(b)\subset
\alpha(a)$. Vice versa, if $\alpha(b)\subset \alpha(a)$, then $b\in \alpha(a)$.
So, $b\le a$.

\begin{defin} For every integer $m$, we define the map $shift_m:I\rightarrow I$,
which takes each class $[i]\in I$ to the class of the binary sequence, obtained from $i$
by adjoining $m$ initial $0$ entries to it if $m\ge 0$ or by deleting $m$
initial entries from it if $m<0$.
\end{defin}

We have that  $shift_0=Id_I$ and $shift_m\circ shift_{-m}=Id_I$ for all
$m$, since changing a finite number of entries in a binary sequence does not
change its class in $I$.

\begin{lem} The maps $shift_m:I\rightarrow I$ induce maps $shift_m:A\rightarrow
A$ and, again, $shift_0=Id_A$ and $shift_m\circ shift_{-m}=Id_A$.
\end{lem}

\section{Properties of lattices $I$ and $A$.}
In this work we need the following crucial property of the image $\alpha(I)$ of
$I$ in the $\sigma$-lattice $A$ under the map $\alpha$. Let $[i_1],[i_2],[i_3],...$ and $[t_1],[t_2],[t_3],...$
be elements of $I$.

\begin{thm}\label{thm.stab} If $$\alpha[i_1]\cup\alpha[i_2]\cup\alpha[i_3]\cup
...=\beta=\alpha[t_1]\cap\alpha[t_2]\cap\alpha[t_3]\cap...$$
for some $\beta\in A$, then there exist some natural numbers $m$ and $n$ such
that
$$\alpha[i_1]\cup...\cup\alpha[i_m]=\beta=\alpha[t_1]\cap...\cap\alpha[t_n].$$ So,
$[i_1]\vee...\vee[i_m]=[t_1]\wedge...\wedge[t_n]$ and $\beta=\alpha([i_1]\vee...\vee[i_m])$.
\end{thm}

\begin{proof} For all natural numbers $p$ and $q$ the inequality $[i_p]\le [t_q]$ must hold.
Thus, $\alpha([i_1]\vee...\vee [i_n])\le
\alpha[t_1]\cap...\cap\alpha[t_n]=\alpha([t_1]\wedge...\wedge [t_n])$ for all
finite $n$.
We define $[i'_n]=[i_1]\vee...\vee [i_n]$ and $[t'_n]=[t_1]\wedge...\wedge
[t_n]$ for all $n$. So $[i'_n]\leq [t'_n]$
for all $n$. For each natural number $n$ we can inductively select
some representatives $i'_n$ and $t'_n$ of the classes $[i'_n]$ and $[t'_n]$ such
that $i'_n\leq t'_n$, $i'_{n-1}\leq i'_n$ and $t'_n\leq t'_{n-1}$ for all $n\geq 2$. Let
$z_n=t'_n-i'_n$ be the binary sequence which has 1 in all the places
where $t'_n$ has 1 and $i'_n$ has $0$ and which has $0$ in all other places.
Note that the classes $[z_1], [z_2], ...$ do not depend on the
choices of the representatives which we made for the classes $[i'_n]$ and
$[t'_n]$, so the subtraction operation $[t'_n]-[i'_n]$ is actually well
defined in $I$. We have the inequalities $z_1\geq z_2\geq ...$. If, after some finite initial $z_1,...,z_{n-1}$, all the following $z_{n},z_{n+2},...$ become almost equal to $(0,0,...)$ then our theorem follows as we will show in the last part of this proof. Otherwise, if all $z_n$ are
not almost equal to $(0,0,...)$, then let $z$ be the binary string which
has its first $1$ in the same place where $z_1$ has its first 1 and
its second $1$ in the same place where $z_2$ has its second 1, and so
on. All the other places of $z$ contain $0$.
\\ \\
Thus, $z$ will have an infinite number
of $1$ entries and for every $n$ almost all these $1$ entries, except some
finite amount of them, will be in the places where $t'_n$ has 1 and
$i'_n$ has $0$. Thus, $\alpha[z]$ is a subset of every $\alpha[t'_n]$.
So, it is a subset of the intersection of all of the $\alpha[t'_n]$, but
$[z]$ is not contained in any of $\alpha[i'_n]$. Therefore, $[z]$ is not contained in
the union of all of the $\alpha[i'_n]$. This contradicts the assumption of our theorem.
\\\\
Hence, for some finite $n$ all
$(z_{n},z_{n+1},...)$ must be almost equal
to $(0,0,...)$. This means that $[i'_n]=[t'_n]$. So
$\alpha([i_1]\vee...\vee [i_n])=\beta=\alpha[t_1]\cap...\cap\alpha[t_n]$. The element $\kappa=[t_1]\wedge...\wedge [t_n]$ of $I$ is contained in
$\beta=\alpha[t_1]\cap...\cap\alpha[t_n]$. So now,
if $\kappa$ is strictly greater than every one of $[i_1],
[i_1]\vee[i_2],[i_1]\vee[i_2]\vee[i_3],...$. This implies that $\kappa$ is not contained in
$\alpha[i_1]\cup\alpha[i_2]\cup\alpha[i_3]\cup ...=\beta$ $-$ which is a contradiction.
\end{proof}
The following corollary of this theorem is crucial for our treatment of signatures
of irregular points, like in Lemma \ref{signatureregular} and in several other places in Section 5:
\begin{cor}\label{cor.finite} If $\alpha[i_1]\cap\alpha[i_2]\cap\alpha[i_3]\cap
...=\alpha[i]$ for some $[i]\in I$ then there exists a finite number $n$
such that $[i_1]\wedge...\wedge[i_n]=[i]$.
\end{cor}

\section{Local topological properties of the P.I.L.}
Let $T$ be a regular, first-countable topological space and $z$ be a point in $T$.
\begin{defin} A sequence of open
neighborhoods $(U(1),U(2),...)$ of $z$ shrinks to $z$ if $U(i+1)\subset U(i)$ for all $i$ and any open neighborhood $V$ of $z$ there exists some $m$ such that $U(m)\subset V$. Thus, the set $\{U(1),U(2),...\}$ is a local base for the topology at $z$.
\end{defin}
Remember, that unless otherwise stated, we assume all the neighborhoods of a point in a Riemann surface to be simply connected. So, in all the cases in which the topological space $T$ is a Riemann surface, all these $U(i)$ are homeomorphic to an open disk. Also, note, that $S_\infty$ is regular and first-countable.
\\ \\
Let $sq=(z(1),z(2),z(3),...)$ be a sequence of points in $S_0$. We
say that a path $p:[0,1]\rightarrow S_0$ passes through the
$sq$ in the correct way if there exist some numbers $0\leq t_1\leq t_2\leq ...\leq
1$ such that $z(m)=p(t_m)$ for all $m=1,2,...$.

\begin{lem}\label{lemma.path} There exists a path $p:[0,1]\rightarrow S_0$ which passes through
$sq$ in the correct way if and only if $sq$ converges to some point $z\in S_0$ and
$$p(\lim_{m\rightarrow\infty}t_m)=z.$$
\end{lem}
\begin{proof} If $sq$ converges to $z$, then we construct the path $p$ as
follows: Let $(U(1),U(2),U(3),...)$ be a sequence of open neighborhoods of
$z$ shrinking to $z$.
Since $S_0$ is Hausdorff and locally compact, these neighborhoods can be
chosen so that each $U_i$ contains all the points
$z(i),z(i+1),...$ and does not contain points $z(1),...,z(i-1)$. We
choose any path inside $U(1)$ from $z(1)$ to $z(2)$, which we regard
as a continuous map from $[0,\frac{1}{2}]$ to $U(1)\subset S_0$.
Next, we choose any path inside $U(2)$ from $z(2)$ to $z(3)$, which we
regard as a continuous map from $[\frac{1}{2},\frac{3}{4}]$ to
$U(1)\subset S_0$. We then proceed to choose a path inside $U(3)$
from $z(3)$ to $z(4)$ and so on. Finally, we glue all these paths
together and we map the point $1\in [0,1]$ to $z$. It is easy to
check that we obtain a continuous function from $[0,1]$ to
$S_0$ and all the points of $sq$ are contained in the image of that function in $S_0$ in the correct way.
\\ \\
For the other direction of the lemma, if there exist some numbers $0\leq
t_1\leq t_2\leq
...\leq 1$ such that $z_i=p(t_i)$ for all $i$, then the sequence
$(t_1,t_2,...)$
converges to some $t\in
[0,1]$ and (since $p$ is continuous) the sequence
$(z(1),z(2),z(3),...)$ converges to $p(t)$.\end{proof}

\begin{lem}\label{lemma.neighborhood} For every irregular point $x\in
S_{\infty}$, there exists some open
neighborhood $U$ of $x$ such that
for any open neighborhood $V$ of $x$ inside $U$, there are infinitely
many positive integers $n(1),n(2),...$ such that $V_{n(i)}$ contains some
critical points of $f_{n(i)-1}$ while $(U-V)_{n(i)}$ does not
contain any critical points of $f_{n(i)-1}$.
\end{lem}
\begin{proof} If this lemma is false, then for any open neighborhood
$V(1)$ of $x$ we can find some open neighborhood $V(2)\subset V(1)$ of $x$ so
that for some infinite sequence $Sq'=(n'(1),n'(2),...)$ of positive integers both
$V(2)_{n'(i)}$ and $(V(1)-V(2))_{n'(i)}$ contain some critical
points of $f_{n'(i)-1}$. By the same logic, for the open neighborhood $V(2)$ of
$x$ we can find some open neighborhood $V(3)$ of $x$ so that for some infinite subsequence $Sq''=(n''(1),n''(2),...)$ of the sequence $Sq'$ all three
sets $V(3)_{n''(i)}$, $(V(2)-V(3))_{n''(i)}$, and $(V(1)-V(2))_{n''(i)}$
contain some critical points of $f_{n''(i)-1}$. This process can
be repeated an infinite number of times. But, since every $f_j$ has at
most $(d-1)\cdot\chi(S_0)$ critical points, where $\chi(S_0)$ is the Euler characteristic of
$S_0$, this process must terminate after at most $(d-1)\cdot\chi(S_0)$ repetitions. Hence, the lemma is true.\end{proof}
The following theorem distinguishes the local topology at the regular and
irregular points. In particular, there is no manifold structure at the irregular
points.

\begin{thm} For every irregular point $x\in S_{\infty}$, there exists some open
neighborhood $U$ of $x$ such that the open
set $U-\{x\}$ has an uncountable number of path-connected components.
\end{thm}

\begin{proof} Let $U$ be an open neighborhood of $x$, as in Lemma
\ref{lemma.neighborhood}, and let\\
$(U(1),U(2),...)$ be a sequence of open neighborhoods of $x$ shrinking to $x$ such that the closure $\overline{U(1)}$ of $U(1)$ is contained in $U$ and $\overline{U(i+1)}$ is contained in $U(i)$ for all $i$. This is possible because $S_\infty$ is regular. So, for any open neighborhood $V$ of $x$, $x$ and $S_\infty-V$ can be separated by two disjoint open sets.
We can find an infinite sequence of increasing positive integers
$(n(1),n(2),...)$ such that $U(i)_{n(i)}$ contains some critical
points of $f_{n(i)-1}$ while $(U-U(i))_{n(i)}$ does not contain any
critical points of $f_{n(i)-1}$. For each $i=1,2,...$ let
$y(i)_{n(i)}\in U(i)_{n(i)}$ be a critical point of $f_{n(i)-1}$
and for all $m=1,...,n(i)-1$ let $y(i)_m=f_{m}\circ ...\circ
f_{n(i)-1}(y(i)_{n(i)})$. Then, for every $n$ the sequence $(y(n)_n,y(n+1)_n,y(n+2)_n,...)$ converges to
$x_n\in S_n$. It can happen, as it does when $x$ corresponds to a super-attracting cycle, that all the $y(i)_{n(i)}$ are equal to $x_{n(i)}$. In that case, $y(n+j)_n=x_n$ for all $n$ and $j$.
\\ \\
By Lemma \ref{lemma.path} there exists a path $p_1:[0,1]\rightarrow S_1$, with
its image entirely inside $U_1$, such that for some sequence of numbers $0\leq t_1\leq
t_2\leq...\leq 1$ and all $i$ we get $p_1(t_i)=y(i)_1$ and $p_1(t)=x_1$,
where $t=\lim\limits_{i\rightarrow\infty}t_i$. If all $y(i)_{n(i)}=x_{n(i)}$ then we just choose a
constant sequence $t_1=t_2=...=t$. Clearly, for any fixed $w\in U-\{x\}$ and
any $u_1\in U_1-\{x_1\}$ we can select this path $p_1$ in such a way that
$p_1(0)=u_1$ and $p_1(1)=w_1$.
\\ \\
We define the path $p_2:[0,1]\rightarrow S_2$ as a lift of the path $p_1$ to $S_2$ such that
$p_2(t)=x_2$ and $p_2(t_i)=y(i)_2$ for all $i$. If more than one such lift for
$p_1$ exists, then we randomly choose one of these lifts to be our $p_2$. We define the path
$p_3:[0,1]\rightarrow S_3$ as a lift of $p_2$ to $S_3$ such that
$p_3(t)=x_3$ and $p_3(t_i)=y(i)_3$ for all $i$. Again, if more than one such lift exists we
randomly choose one of them. We continue this way until we construct the path $p_{n(1)-1}$.
\\ \\
Now, we define two paths $p(0)_{n(1)}:[0,1]\rightarrow S_{n(1)}$ and $p(1)_{n(1)}:[0,1]\rightarrow
S_{n(1)}$ as any two lifts of the path $p_{n(1)-1}$ to $S_{n(1)}$ such that
$p(0)_{n(1)}(t)=p(1)_{n(1)}(t)=x_{n(1)}$ and
$p(0)_{n(1)}(t_i)=p(1)_{n(1)}(t_i)=y(i)_{n(1)}$ for all
$i$, while selecting two different lifts to $S_{n(1)}$ of the piece
$p_{n(1)-1}:[0,t_1)\rightarrow S_{n(1)-1}$. Taking two such
different lifts is possible even in the case when $y(i)_{n(i)}=x_{n(i)}$ and all $t_i=t$. Now, we lift both paths
$p(0)_{n(1)}$ and $p(1)_{n(1)}$ to $S_{n(1)+1}$ while
requiring that their lifts map $t$ to $x_{n(1)+1}$ and
$t_i$ to $y(i)_{n(1)+1}$ for all $i>1$. If for one or both of these paths more than one such lift exists we randomly choose one of these lifts for that path. We continue this way
until we construct the two paths $p(0)_{n(2)-1}$ and $p(1)_{n(2)-1}$.
\\ \\
We define four paths, $p(0,0)_{n(2)}:[0,1]\rightarrow S_{n(2)}$ and
$p(0,1)_{n(2)}:[0,1]\rightarrow S_{n(2)}$ as different lifts of
$p(0)_{n(2)-1}:[0,1]\rightarrow S_{n(2)-1}$ and
$p(1,0)_{n(2)}:[0,1]\rightarrow S_{n(2)}$ and $p(1,1)_{n(2)}:[0,1]\rightarrow S_{n(2)}$ as different lifts of
$p(1)_{n(2)-1}:[0,1]\rightarrow S_{n(2)-1}$ to $S_{n(2)}$, while requiring that
$p(*,*)_{n(2)}$ map $t$ to $x_{n(2)}$ and map $t_i$ to
$y(i)_{n(2)}$ for all $i=2,3,...$. Taking four such different lifts is possible even
when $y(i)_{n(i)}=x_{n(i)}$ and all $t_i=t$ since even in that case the path arcs $p(0)_{n(2)-1}:[0,t_2)\rightarrow S_{n(2)-1}$ and $p(1)_{n(2)-1}:[0,t_2)\rightarrow S_{n(2)-1}$ each admit at least two different lifts to $S_{n(2)}$ and for these four lifts $p(*,*)(t_2)=y_2$. We now lift all four paths to $S_{n(2)+1}$ while
requiring that these lifts map $t$ to $x_{n(2)+1}$ and map $t_i$ to $y(i)_{n(2)+1}$ for all $i>2$.
This process can be continued infinitely.
\\ \\
Thus, for each binary sequence $sq$ of the above mentioned choices of lifts, after possibly
making arbitrary choices in some of our lifts, we obtain a unique lift to $S_{\infty}$ of the path $p_1$. For each $sq$, this unique lift
is contained in $U$, takes $t$ to $x$ and $0$ to the point $x(sq)$ in $U$, where this $x(sq)$ is
some unique point for each binary sequence $sq$ and $x(sq)_1$ is always equal to $u_1$. Thus,
all these points $x(sq)$ are path-connected to $w$ in $U$ and, consequently, are all pairwise path-connected to each other.
Now, fix any binary sequence $sq$ and let $h:[0,1]\rightarrow S_{\infty}$ be any path inside $U$ which does not
contain $x$, such that $h(0)=x(sq)$. Since $[0,1]$ is compact and $S_{\infty}$ is Hausdorff
we can select the sequence $(U(1),U(2),...)$ of neighborhoods of $x$ shrinking to $x$ in such a way that the image of $h$ does not intersect $\overline{U(1)}$. Indeed, the image of a compact set $[0,1]$ under a continuous map $h$ in $S_{\infty}$ is a compact set. In a Hausdorff space $S_{\infty}$, if a compact set does not contain point $x$ then there exists some neighborhood $J$ of $x$, such that $\overline{J}$ is disjoint from that compact set. For any $(U(1),U(2),...)$ shrinking to $x$ we redefine each $U(i)$ to be $U(i)\cap J$.
\\ \\
There is a conformal isomorphism  between $U_1$ and the unit disk, which takes the subset
$U(1)_1$ of $U_1$ to a simply connected domain, whose closure is contained inside that open disk. Hence,
$U_1-\overline{U(1)_1}=(U-\overline{U(1)})_1$ is homeomorphic to an open annulus. So, for some
other binary sequence $sq'$ of path lifts choices, $h(1)=x(sq')$ if and only if the winding number
of the loop $h_1:[0,1]\rightarrow (U-\overline{U(1)})_1$ in the annulus $(U-\overline{U(1)})_1$ around $\overline{U(1)_1}$ is
nonzero. Thus, only a countable amount of different points $x(sq')$ can be connected to the
point $x(sq)$ by a path in $U$ which avoids $x$. So, removing $x$ from $U$ breaks $U$ into
an uncountable number of path-connected components.
\end{proof}
\section{Signature - a local invariant of P.I.L.}
We proceed to introduce an important local invariant of the
P.I.L.
\begin{defin} For an open neighborhood $U\subset S_{\infty}$ and a critical
point $c\in
S_0$, we define the index of $U$ with respect to $c$ to
be the class $ind(U,c)\in I$ of the binary sequence, which has 1 in its $n^{th}$
place if and only if $c\in U_n$.
\end{defin}
It is clear that if $V\subset U$, then $ind(V,c)\le
ind(U,c)$.
\begin{defin} \label{definsignature} For a point $x\in S_{\infty}$ and a critical point $c\in
S_0$ of $f$ we define the signature of $x$ with respect to
$c$ as:
$$sign(x,c)=\bigcap_{j=1}^{\infty}\alpha([ind(U(j),c)]),$$
where $(U(1),U(2),...)$ is an arbitrary sequence of open neighborhoods of $x$ in $S_{\infty}$
shrinking to $x$.
\end{defin}
\begin{lem} \label{signatureinvariant} The signature $sign(x,c)$ does not depend on the choice of the sequence $(U(1),U(2),...)$.
\end{lem}
\begin{proof} Let $(U'(1),U'(2),...)$ be another sequence of open sets
shrinking to $x$. Then, for each $U(j)$ there exists some $U'(j')$ such
that $U'(j')\subset U(j)$. Thus,$$\alpha([ind(U'(j'),c])\subset
\alpha([ind(U(j),c]).$$ So the signature, defined using
$(U'(1),U'(2),...)$, is a subset of the signature, defined using
$(U(1),U(2),...)$. However, by the same argument, the reverse is also true.
Thus, these two signatures are equal. \end{proof}

\begin{lem} \label{signatureregular} A point $x\in S_\infty$ is regular if and only if $sign(x,c)=\{[0,0,0,...]\}$ for every critical point $c$ of $f$.
\end{lem}
\begin{proof} Assume that $x$ is regular. Then, by definition of a regular point, there is a neighborhood $U$ of $x$ such that only a finite number of its projections $U_j$ contain any critical point $c$. But this implies that we can choose
a sequence of open neighborhoods $(U(1),U(2),...)$ of $x$ shrinking to $x$, such that only a finite number of $U(i)_j$ contain any
critical point $c$. Indeed, to obtain such a sequence $(U(1),U(2),...)$ just take any sequence of open neighborhoods of $x$ shrinking to $x$ and intersect all these neighborhoods with $U$. Hence,
$sign(x,c)=\{[0,0,0,...]\}$.
\\ \\
For the other direction of the lemma notice that, by definition of the signature, if
$sign(x,c)=\{[0,0,0,...]\}$ then, for any sequence $(U(1),U(2),...)$ of open neighborhoods of $x$ shrinking to $x$,
$\bigcap_{i=1}^{\infty}\alpha([ind(U(i),c)])=\alpha([0,0,0,...])$. Hence, by Corollary \ref{cor.finite}, there exists some finite
$i$ such that $[ind(U(i),c)]=[0,0,0,...]$. Thus, $c$ belongs to, at most, a finite number of
the projections $U(i)_j$ of $U(i)$, where $j=1,2,...$. Therefore, if $sign(x,c)=\{[0,0,0,...]\}$ for all the critical
points $c$ of $f$ then for some neighborhood $U$ of $x$ only a finite number of its projections $U_j$ will contain a critical point of $f$. This implies that $x$ is regular.
\end{proof}

\begin{lem}\label{lemma.signatures} For any $x,x'\in S_{\infty}$, if
$sign(x,c)\cap sign(x',c)$ contains any element other than
$[0,0,0,...]$, then $x=x'$.
\end{lem}
\begin{proof} If $x\ne x'$ then there exists some integer $t$
such that $x_t\ne x'_t$. We can find disjoint neighborhoods
$U_t$ of $x_t$ and $U'_t$ of $x'_t$ and construct
neighborhoods $U$ of $x$ and $U'$ of $x'$, so that $U_j$ and $U'_j$
are disjoint for all $j\ge t$. Hence, $ind(U,c)\cap
ind(U',c) =\{[0,0,0,...]\}$. \end{proof}

\begin{lem} For any integer $m$  and any point $x\in S_{\infty}$, we have:
$$sign(f^m(x),c)=shift_{-m}(sign(x,c)).$$\end{lem}

Now we investigate irregular points of P.I.L. and study their
local properties. Again, let $S_0$ be the Riemann sphere, the complex plane or the unit disk. Let us recall the following definition:

\begin{defin} \label{postcriticdef} For a critical point $c$ of $f$, its $\omega$-limit set
$\omega(c)$ is the set of all the points $x_0\in S_0$, for which there exists some sequence $(i_1,i_2,...)$ of increasing positive integers, such that the sequence $(f^{i_1}(c),f^{i_2}(c),...)$ converges to $x_0$. We permit some or all of $f^{i_t}(c)$ in this sequence to be equal to $x_0$.
\end{defin}
When $c$ is not periodic or pre-periodic, $\omega(c)$ is identical to the set of all the accumulation points of the forward orbits of $c$. When $c$ is periodic or pre-periodic, $\omega(c)$ is the set of all the points of that periodic cycle. The set $\omega(c)$ is forward invariant, meaning $f(\omega(c))\subset \omega(c)$. But, in general, $\omega(c)$ is a proper subset of $f^{-1}(\omega(c))$. For any point $x\in S_{\infty}$ and any critical point $c$, the signature $sign(x,c)$ may contain an element different from $[0,0,...]$ only if
all the coordinates $x_i$ of $x$ belong to $\omega(c)$. Recall the following properties of the $\omega$-limit set:
\begin{lem} The set $\omega(c)$ is closed. If $\omega(c)$ is not empty, then the map $f$, restricted to $\omega(c)$, is a surjection of $\omega(c)$ onto itself.
\end{lem}
\begin{proof} Take any convergent sequence $(x_1,x_2,...)$ of points of $\omega(c)$. Then, for each $x_j$ we have some sequence of increasing positive integers $(i_{j,1},i_{j,2},...)$ as in Definition \ref{postcriticdef}, above. Define $i_1=i_{1,1},i_2=i_{2,2},...$. The sequence $(f^{i_1}(c),f^{i_2}(c),...)$ converges to the limit point of the sequence $(x_1,x_2,...)$. So, $\omega(c)$ contains that limit point. Since $S_0$ is first-countable, $\omega(c)$ is a closed set.
\\ \\
Take any $x\in S_0$ and let $y_1,...,y_t$ be all the pre-images of $x$ in $S_0$ under $f$. If each $y_i$ has some neighborhood $U_i$ in $S_0$ such that $U_i$ contains no $f^k(c)$ for all $k$ larger than some $m_i$, then, since $f$ is an open self-map of $S_0$, the open neighborhood $f(U_1)\cap...\cap f(U_t)$ of $x$ will not contain $f^k(c)$ for all $k$ larger than $1+\max(m_1,...,m_t)$. So, $x$ is not in $\omega(c)$. Hence, for every $x\in \omega(c)$, at least one of its pre-images under $f$ is also in $\omega(c)$.
\end{proof}
Notice that if, for example, $S_0$ is the complex plane and $f(z)=z^2+1$ then the only critical point of $f$ is $c=0$ and $\omega(c)$ is empty.
\begin{defin} A subset $V$ of $\omega(c)$ is called inverse-critical with
respect to $c$ if, for any $z\in V$ and any neighborhood $U$ of $z$, we can
always find a pre-image $y$ of $z$ in $V$ under some iterate $f^n$ of $f$ such that the connected component of $f^{-n}(U)$ containing
$y$ also contains $c$.
\end{defin}
Notice that a point $x\in S_\infty$ is irregular if and only if the set of its coordinates is inverse-critical with respect to at least one critical point $c$.
\\ \\
If a set $V$ is inverse-critical with respect to $c$, then the set $f(V)$ is inverse-critical with respect to $c$. Also, if $V$ and $W$ are
inverse-critical with respect to $c$, then $V\cup W$ is
inverse-critical with respect to $c$. Thus, we define the following:
\begin{defin} The subset $\Gamma(c)$ of $\omega(c)$, which is the union of all
the inverse-critical sets with respect to $c$, is called the maximal
inverse-critical set with respect to $c$.
\end{defin}
Notice that if $\Gamma(c)$ is not empty, then $f(\Gamma(c))=\Gamma(c)$.
\\ \\
Recall (see, for example, Paragraph 3 in \cite{Lyubich2} and page 7 in \cite{CK2}) that a quadratic polynomial $f(z)=z^2+a$ acting on $\mathbb{C}$ is called persistently recurrent, if the critical point $c=0$ is not periodic and not pre-periodic and all the points of the invariant lift of $\omega(c)$ to the P.I.L. are irregular. In this case $\Gamma(c)=\omega(c)$.

\begin{thm} \label{nontrivsign}
For a point $x\in S_{\infty}$ and a critical point $c$, the signature $sign(x,c)$ can be greater than $\{[0,0,...]\}$ only if all $x_i$ are contained in $\Gamma(c)$. Vice-versa, for any $x_1\in \Gamma(c)$ we can, starting with this $x_1$,
construct a point $x=(x_1,x_2,...)$ in $S_{\infty}$ such that $sign(x,c)$ is greater than $\{[0,0,...]\}$. In other words, the invariant lift of $\Gamma(c)$ to the P.I.L. is the set of all irregular points $x$, such that $sign(x,c)$ is not trivial.
\end{thm}
\begin{proof}

If some $x_i$ is not contained in $\Gamma(c)$ while $sign(x,c)$ is greater than $\{[0,0,...]\}$, then adding the points $x_i,x_{i+1},...$ to $\Gamma(c)$ creates a new inverse-critical set with respect to $c$, which is a proper superset of $\Gamma(c)$. This constitutes a contradiction to the maximality of $\Gamma(c)$.
\\ \\
Let $(U(1)_1,U(2)_1,U(3)_1,...)$ be a sequence of neighborhoods of
$x_1$ shrinking to $x_1$. We select pre-images $x_2\in \Gamma(c)$ of $x_1$, $x_3\in \Gamma(c)$ of $x_2$, and so on, until some pre-image $x_{k_1}$ of $x_{k_1-1}$, in such a way that the lift $U(1)_{k_1}$ of $U(1)_1$ along these
pre-images contains $c$.
\\ \\
Again, we select a sequence of pre-images of $x_{k_1}$ until some pre-image
$x_{k_2}\in \Gamma(c)$ of $x_{k_1}$, such that the lift $U(2)_{k_2}$ of $U(2)_1$
along these pre-images contains $c$.
\\ \\
This process can be continued infinitely, thus producing some $x=(x_1,x_2,...)$ in
$S_{\infty}$. Clearly, all the neighborhoods $U(t)=(U(t)_1,U(t)_2,...)$ of $x$, for $t=1,2,...$, contain an infinite amount
of copies of $c$. Also, $(U(1),U(2),...)$ shrink to $x$. Thus, $x$ is irregular and
$sign(x,c)> \{[0,0,...]\}$.
\end{proof}
\begin{defin} \label{patched} Let $V$ be any subset of $S_0$. An open set $U\subset S_0$, which contains $V$, is called a patched neighborhood of $V$ if to each $x\in V$ an open neighborhood $U_x$ of $x$ in $S_0$ is associated in such a way that $U$ is the union of all these $U_x$.
\end{defin}
We can always construct a patched neighborhood $U$ of a set $V$ by selecting for every point $x\in \omega(c)$ some neighborhood $U_x$ of $x$ in $S_0$ and defining: $$U=\bigcup\limits_{x\in V} U_x.$$
Notice that by our requirement in this work (always considering only simply connected neighborhoods of points in a Riemann surface), each $U_x$ must be simply connected, but a patched neighborhood $U$ of $\omega(c)$ does not have to be simply connected or connected. In this article we will consider patched neighborhoods only in the case, in which $V=\omega(c)$.
\\ \\
Let $U$ be a patched neighborhood of $\omega(c)$.
\begin{defin} The open set $\bigcup\limits_{x\in \omega(c)} f^{-1}(U_x)$, in which each $f^{-1}(U_x)$ is taken along all the pre-images of $x$ in $\omega(c)$ and only along them, is called the pre-image $U^{-1}$ of a patched neighborhood $U$ of $\omega(c)$ along $\omega(c)$. We define the pre-images $U^{-2},U^{-3},...$ of $U$ along $\omega(c)$ in a similar manner. Clearly, all $U^{-1},U^{-2},...$ are, in a trivial way, also patched neighborhoods of $\omega(c)$.
\end{defin}
For a patched neighborhood $U$ of $\omega(c)$ we define its ``derived set" $U'$ as the set of all the points $y$ in $S_0$ such that for this $y$ there exists some non-negative integer $k$ so that $y$ is contained in $\bigcap\limits_{i=k}^{\infty} U^{-i}$. In other words, $$U'=\bigcup\limits_{k=0}^{\infty}\left(\bigcap\limits_{i=k}^{\infty} U^{-i}\right).$$ Clearly, $\omega(c)\subset U'$. Note that $U'$ does not have to be an open set. Note that $f(U')$ is a subset of $U'$, since if for some point $y$ and integer $k$ (we can always increase $k$ so that $k>0$), $y$ is contained in $\bigcap\limits_{i=k}^{\infty} U^{-i}$, then $f(y)$ is contained in $\bigcap\limits_{i=k-1}^{\infty} U^{-i}$.
\begin{defin} \label{thickpostcritic} The intersection of all $U'$ for all the patched neighborhoods $U$ of $\omega(c)$ is called the thickening $\Omega(c)$ of $\omega(c)$. Clearly, $\omega(c)\subset \Omega(c)$.
\end{defin}
Intuitively, the concept of thickening $\Omega(c)$ of $\omega(c)$ is a generalization of the concept of the immediate basin of attraction for the cases of the attracting, super-attracting, and parabolic cycles. In these cases, $\omega(c)$ is precisely the set of all the points of the cycle and $\Omega(c)$ is the whole immediate basin of attraction, as we will briefly discuss after the proof of Theorem \ref{periodicsignatures}. In Theorem \ref{postcriticalirreg}, below, we state a sufficient condition for $\Gamma(c)$ to be non-empty. Moreover, under this condition $\Gamma(c)=\omega(c)$. Some examples of when this condition is satisfied, are the cases of the attracting, super-attracting and parabolic cycles, described in Theorem \ref{periodicsignatures}, and the non-cyclical case, described in Lemma \ref{noncycleirreg} and in what follows it. Another non-cyclical example of a case, when this condition is satisfied, is provided immediately after the proof of Theorem \ref{postcriticalirreg}.
\begin{thm} \label{postcriticalirreg}
If $c\in \Omega(c)$, then $\Gamma(c)=\omega(c)$.
\end{thm}
\begin{proof}
Take any arbitrary $x\in \omega(c)$ and take any neighborhood $U_{x}$ of $x$ in $S_0$. To demonstrate that $x\in \Gamma(c)$ we need to show that there exists some positive integer $t$ and some pre-images $f^{-1}(x)\in\omega(c), ..., f^{-t}(x)=x'\in\omega(c)$ of $x$ so that lifting $U$ along these pre-images will produce a neighborhood $f^{-t}(U_x)$ of $x'$ which contains $c$. Indeed, if this is true for any arbitrary $x\in \omega(c)$ then it can be repeated again for the pre-image $x'\in\omega(c)$ of $x$, and so on.
\\ \\
Since $S_0$ is regular we can find two disjoint open sets $V_1$ and $V_2$ in $S_0$ such that $V_1$ contains $\omega(c)-U_x$ and $V_2$ is a neighborhood of $x$ whose closure in $S_0$ is contained in $U_x$.
\\ \\
Now, construct a patched neighborhood $U$ of $\omega(c)$ by associating to every point $y\in \omega(c)\cap U_x$ the neighborhood $U_x$ as its $U_y$ and associating to every other point $y$ of $\omega(c)$ any neighborhood $U_y$ of $y$ which is contained in $V_1$. In this way, every neighborhood $U_y$, associated to $y\notin U_x$, is disjoint from $V_2$.
\\ \\
Since $x\in \omega(c)$ there exists some infinite increasing sequence of positive integers $(t_1,t_2,...)$ such that $f^{t_i}(c)$ is inside $V_2$ for all $i$. Also, by definition of $\Omega(c)$, for our $U$ there exists some non-negative integer $k$ such that $c\in U^{-t}$ for all $t\geq k$. Now, take $t$ to be any one of the integers $t_i$ which is greater than or equal to $k$. Then $c\in U^{-t}$ and $f^t(c)\in V_2$, so $f^t(c)\notin U-V_2$. Hence, $c$ must be contained in the $t^{th}$ lift of $U_x$ along some pre-images of $x$ which are contained in $\omega(c)$. Therefore, $x\in \Gamma(c)$.
\end{proof}
From Proposition 2.3 of \cite{Lyubich} (on page 99) we deduce the existence of examples of maps $f_w(z)=1+\frac{w}{z^2}$, with a complex parameter $w$, such that $\omega(c)$, for $c=\infty$, is the entire Riemann sphere. In this case $\Omega(c)$ will also be the entire Riemann sphere. Thus, the condition of Theorem \ref{postcriticalirreg} is satisfied and $\Gamma(c)=\omega(c)$. Hence, by Theorem \ref{nontrivsign}, in these cases we, starting from any point $x_1$ on the Riemann sphere, can, by taking some appropriate pre-images of $x_1$, construct an irregular point $x=(x_1,x_2,...)$. Additionally, since for any map $f_w(z)=1+\frac{w}{z^2}$, the only pre-image of the critical point $\infty$ under $f_w$ is the critical point $0$, the equality $sign(x,0)= shift_1(sign(x,\infty))$ is satisfied for all the irregular points $x$ of all $f_w(z)$. This equality is, obviously, satisfied for all the regular points.
\begin{thm}
Let $x=(x_1,x_2,...)$ be an irregular point. If there exists some integer $t$ such that for any two different integers $i$ and $j$ greater than $t$ we have $x_j\ne x_i$ (which means that $x$ is not a lift of a cycle to the inverse limit), then $sign(x,c)\cap shift_k(sign(x,c))=\{[0,0,...,0]\}$ for all critical points $c$ and all nonzero integers $k$.
\end{thm}
\begin{proof}
Define the irregular point $y$ by $y_{j+k}=x_j$ for all $j=1,2,...$,
and $y_k=f(y_{k+1})$, $y_{k-1}=f(y_{k})$,$...$, $y_1=f(y_2)$. Then, since all the $x_i$ with $i>t$ are different, $y$ is not
equal to $x$. So, by Lemma \ref{lemma.signatures}, \\

$sign(x,c)\cap shift_k(sign(x,c))=sign(x,c)\cap sign(y,c)=\{[0,0,...,0]\}$.
\end{proof}
Now we  briefly review some basic concepts from holomorphic
dynamics. See \cite{Bea}, \cite{CG} and \cite{Milnor} for a concise treatment of
the subject.
Let $f:S_0\rightarrow S_0$, where $S_0$ is the Riemann sphere, the complex plane or the unit disk, be a rational function. A periodic cycle $C$
is a set of pairwise different points $\{x_1,...,x_n\}\subset S_0$ such that $f(x_{i+1})=x_{i}$ for
all $i=1,...,n$ and $f(x_1)=x_n$. The number $n$ is called the period of the cycle. The multiplier
$\lambda=(f^n)'(x_0)$ is a conjugacy invariant associated with every
periodic cycle. A periodic cycle $C$ is called:
\begin{itemize}
\item \textit{repelling} if $|\lambda|>1$;
\item \textit{attracting} if $0<|\lambda|<1$;
\item \textit{super-attracting} if $\lambda=0$; and
\item \textit{neutral} if $|\lambda|=1$.
\end{itemize}
If a cycle is either attracting or repelling then around any point of that cycle the map $f^n$ is locally
conjugated to the map $z\mapsto \lambda z$.
If a cycle is super-attracting then around any point of that cycle the map $f^n$ is locally
conjugated to the map $z\mapsto z^m$, where $m$ is the degree of $f^n$
at any point of the cycle.
\\ \\
The neutral cycles are subclassified as follows:
\\
Let the multiplier
$\lambda$ of the cycle $C$ be $e^{2\pi \theta i}$ where $\theta$ is called the rotation number of the cycle. Then a neutral cycle $C$ is called:
\begin{itemize}
\item \textit{parabolic} if $\theta$ is rational;
\item \textit{Siegel} if $f^n$ is locally conjugated to
the map $z\mapsto e^{2\pi \theta i}z$; or
\item \textit{Cremer} if there is no such conjugation, as in
the Siegel case.
\end{itemize}
In a parabolic cycle, while around any point of that cycle the map $f^n$ can not be locally conjugated to the map $z\mapsto z+1$, such a conjugation is still possible ``in pieces", as described by the Leau-Fatou Flower Theorem.
\begin{lem} \label{structureofomegaprime} If for some critical point $c$, the set $\Omega(c)$ contains an open disk $D$, but is not the entire Riemann sphere, complex plane, or complex plane with one point removed, then $\omega(c)$ is a super-attracting, attracting, or parabolic cycle, and $\Omega(c)$ is the whole immediate basin of attraction of that cycle.
\end{lem}
\begin{proof}
Let $U$ be any patched neighborhood of $\omega(c)$ and $U'$ be its derived set. Then $f(U')$ is a subset of $U'$. Hence, $f$ maps $\Omega(c)$ into itself. Hence, if $\Omega(c)$ is not the entire Riemann sphere, the complex plane, or the complex plane with one point removed, but still contains some open disk $D$, then the forward iterates of $f$ restricted to $D$ form a normal family since their images are contained in $\Omega(c)$. Therefore, $D$ must be contained in the Fatou set of $f$. By Sullivan's No Wandering Domain Theorem (see \cite{Milnor}, page 259, and \cite{CG}, pages 69 and 70), for some non-negative integer $k$, the open disk $D'=f^k(D)$ (which is also contained in $\Omega(c)$) will belong to a periodic Fatou component $W$ of $f$.
\\ \\
Suppose, that $W$ is a Siegel disk or a Herman ring. Select any $x\in D'$. Then the forward orbits $f(x),f^2(x),...$ are dense in a Jordan circle $J_x$ inside $W$. If the closed set $J_x$ does not intersect the closed set $\omega(c)$, then we can find two disjoint open sets $B$ and $B'$, such that $B$ contains $\omega(c)$ and $B'$ contains $J_x$. This contradicts the fact that for any patched neighborhood $U$ of $\omega(c)$, the forward orbits of $x$ under $f$, starting from some iteration of $f$ and onward, must all belong to $U$. Hence, $\omega(c)$ must intersect $J_x$. This implies that for some positive integer $m$, $f^m(c)$ belongs to $W$. But this, in turn, implies that $\omega(c)$ is a Jordan disk $J$ inside $W$. Now we select some $y\in D'$, which does not belong to $J$. Forward orbits of $y$ under $f$ are dense in a Jordan disk $J_y$ in $W$. Closed sets $J_y$ and $J=\omega(c)$ are disjoint. Thus, they can be separated by some disjoint open sets. This contradicts the fact that for any patched neighborhood $U$ of $\omega(c)$, the forward orbits of $y$ under $f$ starting from some iteration of $f$ and onward, must all belong to $U$.
\\ \\
Hence, $W$ must be a super-attracting, attracting, or parabolic periodic Fatou component. This implies that $\omega(c)$ is the corresponding super-attracting, attracting, or parabolic cycle and $\Omega(c)$ is the whole immediate basin of attraction of that cycle.
\end{proof}
Let $x=(x_1,x_2,...)\in S_{\infty}$ be the lift of the cycle
$\{x_1,...,x_n\}$ of the dynamical system to $S_{\infty}$. In other words,
$x_i=x_{i+n}=x_{i+2n}=...$ for all $1\le i\le n$.
\\ \\
The following theorem relates dynamical properties of a cycle and the signature of its lift to the P.I.L. The part of this theorem, dealing with the invariant lifts of Cremer and Siegel cycles, was also stated, in a different manner, and proven in \cite{CK}. Here we
present another proof.
\begin{thm}\label{periodicsignatures} If $x$ is the invariant lift of:
\begin{enumerate}
\item a repelling cycle or a Siegel cycle, then $x$ is a regular
point. So, $sign(x,c)=\{[0,0,...]\}$ for every critical point $c$;
\item either an attracting cycle or a
super-attracting cycle or a parabolic cycle or a Cremer cycle, then $x$ is an
irregular point. Thus, by Lemma \ref{signatureregular}, it must have a non-trivial signature with
respect to some critical point. With respect to every critical point $c$,
the signature $sign(x,c)=shift_{\pm n}(sign(x,c))$.
\begin{enumerate}
\item In the attracting and super-attracting cases the
signature of $x$, with respect to some critical points $c_1,...,c_m$, is:
$$sign(x,c_j)=shift_{k_j}(\alpha[sq(n)]),$$ where $0\le k_j<n$ are
some integers and the binary sequence $sq(n)$ has 1 in the places
$n,2n,3n,...$ and $0$ in all the other places. With respect to all the
other critical points, the signature of $x$ is $\{[0,0,...]\}$.
\item In the parabolic case the
signature of $x$, with respect to some critical point $c$, is:
$$sign(x,c)=shift_{k}(\alpha[sq(n)]),$$ where $0\le k<n$ is some integer
and the binary sequence $sq(n)$ has 1 in places $n,2n,3n,...$ and
$0$ in all the other places.
\end{enumerate}
\end{enumerate}
\end{thm}
\begin{proof} If $x$ represents a repelling cycle, or a Siegel cycle, then $x_1$
has some neighborhood $U_1\subset S_1$ such that the pre-images $U_2,...,U_n$ of
$U_1$ along $x_2,....,x_n$ do not contain any critical points and each
$U_{i+n}$ is a subset of $U_i$ for all $i$. Thus, none of $U_i$ will contain any
critical points. Therefore, $sign(x,c)=\{[0,0,...]\}$ for every critical
point $c$.
\\ \\
If $x$ represents an attracting or a super-attracting cycle, then (see
\cite{Milnor}, \cite{CG}, \cite{Bea}) for any ``small-enough" neighborhood $U_1$ of $x_1$ there exists
some positive integer $q$ such that for some critical point $c_1$ and some
integer $1\le k_1\le n$ all the pre-images $U_{t\cdot n+
k_1}$ of $U_1$ along $x$, where $t\geq q$, will contain $c_1$, while the pre-images
$U_i$ and $U_j$ of $U_1$ along $x$ will all be pairwise disjoint if $(i\ne j)\;\mod\; n$.
Hence, we get that $sign(x,c_1)=shift_{k_1}(\alpha[sq(n)])$. But, if
$k_1=n$ we can substitute $shift_{n}(\alpha[sq(n)])$ with $shift_{0}(\alpha[sq(n)])$.
\\ \\
On the other hand, let $c_2$ be any critical point such that $sign(x,c_2)$ is non-trivial. Then $c_2$ is contained in
the immediate basin of attraction of the cycle. Hence, for any ``small-enough" neighborhood $U_1$
of $x_1$ there exists some integer $q$ such that all the
pre-images $U_{t\cdot n+q}$ of $U_1$ along $x$ contain $c_2$ for all $t=0,1,2...$ while all other pre-images of $U_1$ along $x$ do not contain $c_2$. Let $(k_2=q)\;\mod\; n$. So, we get $sign(x,c_2)=shift_{k_2}(\alpha[sq(n)])$.
\\ \\
If $x$ is the invariant lift of a parabolic cycle (see Theorem 10.15 in \cite{Milnor}
and see \cite{Berg}) then some immediate attracting basin of some $x_k$,
where $1\le k\le n$, must contain a critical point $c$. Thus, just like in the attracting and
super-attracting cases, for any ``small-enough" neighborhood $U_1$ of $x_1$ there exists some
positive integer $q$ such that all the pre-images $U_{t\cdot n+
k}$ of $U_1$ along $x$, where $t\geq q$, will contain $c$, while the pre-images $U_i$
and $U_j$ of $U_1$ along $x$ will all be pairwise disjoint if $(i\ne j)\;\mod\; n$. Hence, we get $sign(x,c)=shift_{k}(\alpha[sq(n)])$. But, if $k=n$ we can substitute $shift_{n}(\alpha[sq(n)])$ with $shift_{0}(\alpha[sq(n)])$. Note that
it follows from the results of \cite{Berg} that there are certain correlations between the number of
critical points, with respect to which $x$ must have such a periodic signature, and the
dynamical properties of the parabolic cycle.
\\ \\
If $x$ is the invariant lift of a Cremer cycle, then let $\Upsilon$ be any
neighborhood of $x$ in $S_\infty$.
We need to show that for some $c$ and for an infinite number of different
positive integers $i$, the point $c$ belongs to the $i^{th}$ projection $\Upsilon_i$ of
$\Upsilon$. If this is not the case, we can
select some small enough $\Upsilon$ such that none of $\Upsilon_i$ contains any critical points.
Let $U(i)$ be the copy of $\Upsilon_i$ in $S_0$.  All the sets
$U(i),U(i+n),U(i+2n),U(i+3n),...$ contain the common point $x(i)$ $-$ the copy of $x_i$ in $S_0$.
Hence, $U=U(i)\cap U(i+n)\cap U(i+2n)\cap ...$ is a path-connected open subset of $S_0$. Let $V=f^n(U)$. Then, $V$
is also a path-connected open set and $U$ is a subset of $V$. Since $f^n$ is a
holomorphic function and the absolute value of its derivative at $x(i)$ is $1$ we can always select a ``small enough" $\Upsilon$ so that $U(i)$ is a full path-connected
component of $f^{-n}(f^n(U(i))$. By its definition $U(i+t\cdot n)$ is a full path-connected component
of $f^{-n}(U(i+(t-1)\cdot n))$. Hence, $U$ is a full path-connected component of $f^{-n}(V)$.
Now, by our construction, $U$ does not contain any critical points of $f$ and, consequently,
of $f^n$. Therefore, $f^n:U\rightarrow V$ is a covering map of some degree $m$ between
$1$ and $d^n$.
\\ \\
If $U$ is the whole Riemann sphere then $V=U$ and, since $f^n$ has no critical points
in $U$, the degree $d$ of $f$ must be $1$. If $U$ is the Riemann sphere with one point
removed then $V$ must be equal to $U$ and $f^n$, restricted to $U$, is conjugated to a
linear function on the complex plane. So, it can not have a Cremer fixed point.
\\ \\
If $U$ is just the whole Riemann sphere with two points removed then
$V=U$ and $f^n$ is conjugated to $z^{\pm m}$ and has no Cremer points.
So, $U$ and $V$ are conformally hyperbolic. If $U=V$ then there exists a conformal
covering map $\psi$ from the unit disk on $U$, which takes $0$ to $x(i)$. Then
the map $\psi^{-1}\circ f^n\circ \psi$ from the unit disk onto itself, which is constructed by
 requiring that $\psi^{-1}$ takes $x(i)$ to $0$, is a
well-defined conformal automorphism of the unit disk which maps $0$
to itself. Hence, it must be a rotation of a unit disk - the multiplication by
$\lambda$. Consequently, in some neighborhood of $x(i)$ in $S_0$ the function
$f^n$ is conjugated to a rotation and $x_i$ is not a Cremer point of
$f^n$.
\\ \\
Therefore, $U$ is strictly smaller than $V$. As a corollary of the
Schwarz-Pick-Ahlfors Theorem (see pages 22-24 in \cite{Milnor}), the inclusion
map of $U$ into $V$ strictly decreases the hyperbolic distance
$dist_U(x,y)>dist_V(x,y)$ for all the ``close-enough" points $x$ and $y$, $x\ne
y$, in $U$. On the other hand, by the Schwarz-Pick-Ahlfors Theorem, the function
$f^n:U\rightarrow V$ is a local isometry and for
all the ``close-enough" points $x$ and $y$ in $U$ we have
$dist_V(f^n(x),f^n(y))=dist_U(x,y)$. Thus, for all the ``close-enough" points
$x$ and $y$, $x\ne y$, in $U$ we get $dist_V(f^n(x),f^n(y))>dist_V(x,y)$.
So, $f^n$ is repelling near $x_i$, which again contradicts the fact that $x_i$
is a Cremer point of $f^n$. Thus, for some critical point $c$ and for any
neighborhood $U$ of $x$, $c$ is contained in an infinitely many projections
$U_i$ of $U$. So, from the Corollary \ref{cor.finite}, repeating the argument
in the proof of Lemma \ref{signatureregular}, we get that the signature $sign(x,c)$ is greater
than $\{[0,0,...]\}$ for that critical point $c$.
\\ \\
Now we prove that for every critical point $c$ in all the cases of our theorem, $sign(x,c)=shift_{\pm n}(sign(x,c))$. This is clearly true for all critical points $c$ such that the signature $sign(x,c)=\{[0,0,...]\}$. Now, take any $c$ such that the signature
$sign(x,c)$ is greater than $\{[0,0,...]\}$ and, consequently, contains some non-zero element $[a]$ of $I$.  Let
a binary sequence $a$ be a representative of the class $[a]$ and $(U(1),U(2),...)$ be a sequence of neighborhoods of $x$ which shrink to $x$.
Then, for each $U(t)$ and for almost every  $i=1,2,...$ we have that - if the $i^{th}$ entry of $a$ is
$1$ then the $i^{th}$ projection $U(t)_i$ of $U(t)$ contains the critical point $c$.
\\ \\
Let $V(t)$, for all $t=1,2,...$, be as follows: $V(t)_i=U(t)_{i+n}$
is regarded as a neighborhood of $x_i$ in $S_i$ (since $x_i=x_{i+n}$) for $i=1,2,...$. Then the
sequence $(V(1),V(2),...)$, just like the original sequence $(U(1),U(2),...)$, is also a
sequence of neighborhoods of $x$. Since $U(t)_i=f^n(U(t)_{i+n})$, we
have that $f^n((V(t)))=U(t)$ for all $t$. So, the sequence $(V(1),V(2),...)$ also shrinks
to $x$. Indeed, $f^n$ is a continuous self-bijection of $S_\infty$. So, if
some open neighborhood of $x$ does not contain any of $V(t)$ as its subset then $f^{-n}$ of that neighborhood,
which also is a neighborhood of $x$, will not contain any of $U(t)$ as its subset.
Lemma \ref{signatureinvariant} states
that the signature does not depend on the choice of the sequence of neighborhoods
of $x$ shrinking to $x$. Hence, $(U(1),U(2),...)$ and $(V(1),V(2),...)$ must produce
the same $sign(x,c)$.
\\ \\
For almost every $i=1,2,...$ we have that if the $i^{th}$ entry of $a$ is $1$ then, by definition of
signature, the $i^{th}$ projection $V(t)_i$ of $V(t)$, for every $t$,
contains the critical point $c$. Since $V(t)_i$ is just the copy of $U(t)_{i+n}$ in
$S_i$, we get that for almost every $i=1,2,...$ the projection $U(t)_{i+n}$ of $U(t)$ in
$S_{i+n}$ must also contain the critical point $c$ for every $t$. Hence, $shift_{n}([a])$
(and by the same logic $shift_{-n}([a])$) must also be contained in the signature
$sign(x,c)$. Since this is true for every element $[a]$ of $sign(x,c)$ we get that
$sign(x,c)=shift_{\pm n}(sign(x,c))$.
\end{proof}
The proof of the part of this theorem, which deals with the Cremer cycles, is very similar to the proof of the first part of Theorem 11.17 (on page 138) of \cite{Milnor}.
Indeed, from the statement of our theorem that an invariant lift of a Cremer cycle is an irregular point, the first part of Theorem 11.17 in \cite{Milnor}, which states that a Cremer cycle is a limit cycle of some critical points of $f$, immediately follows.
\\ \\
Note that for the super-attracting, attracting and parabolic cycle cases, there exists a critical point $c$ such that its forward orbits are either points of the cycle (super-attracting case) or converge to the cycle (attracting and parabolic cases). In these three cases, for that critical point $c$ the set $\omega(c)$ is exactly the set of the points of the cycle and $\Omega(c)$ is the immediate basin of attraction of that cycle. For the super-attracting and attracting cycle cases this will be the situation for every critical point $c$, for which the signature of the lift of the cycle is non-trivial.
\\ \\
It is currently unknown to us if there can exist a critical point $c$ such that $\omega(c)$ contains a parabolic cycle as its proper subset. The situation in the Cremer case is currently even more unclear to us. Thus, we do not even know if in the case of a Cremer cycle there always must exist a critical point $c$ such that $\omega(c)$ is equal to the set of the points of the cycle.
\\ \\
Let us now provide one important non-cyclical case, for which one irregular point can be explicitly constructed and its signature can be explicitly computed. We refer to \cite{McMullen}, \cite{Lyubich3}, \cite{Jiang} and \cite{CK2} for more details relevant to this case.
\\ \\
A quadratics self-map $f(z)=z^2+a$ of the complex plane $\mathbb{C}$, where $a$ is a fixed complex number, is called infinitely renormalizable, if there exists an infinite sequence $renseq=(\{f(i),U(i),V(i),n(i)\}_{i=1,2,...})$ in which:
\begin{itemize}
\item $n(1),n(2),...$ are increasing positive integers and each $n(i)$ divides $n(i+1)$;
\item $U(1),U(2),...$ and $V(1),V(2),...$ are simply connected, open neighborhoods of $0$, isomorphic to a disk, such that for $i=1,2,...$ the closure $\overline{U(i)}$ of $U(i)$ is contained in $V(i)$ and $V(i+1)$ is contained in $U(i)$; and
\item $f(i):U(i)\rightarrow V(i)$ is the restriction of $f^{n(i)}$ to $U(i)$ for $i=1,2,...$. Each $f(i)$ is a branched covering of $V(i)$ of degree two.
\end{itemize}
We say that an infinitely renormalizable map $f(z)=z^2+a$ has \textit{a priori} bounds, if for some number $\delta>0$ the neighborhoods $U(i)$ and $V(i)$ of $0$ can be chosen in such a way, that the modulus $mod(\overline{V(i)}-U(i))$ of the closed annulus $\overline{V(i)}-U(i)$ is greater than $\delta$ for all $i=1,2,...$.
\\ \\
Let $f(z)$ be an infinitely renormalizable map, such that the sequence\\ $(U(1),U(2),...)$ of neighborhoods of $0$ converges to $0$. This, for example, is the case when the infinitely renormalizable map $f(z)$ has \textit{a priori} bounds. Since $V(i+1)$ is contained in $U(i)$ for all $i$, we get that the sequence $(V(1),V(2),...)$ of neighborhoods of $0$ also converges to $0$.
\begin{prop} In the plaque inverse limit $S_{\infty}$ of $f:\mathbb{C}\rightarrow \mathbb{C}$ there exists exactly one point $x=(x_1,x_2,...)$ such that $x_1=0$ and each $x_{n(i)+1}$ is a pre-image of $x_1$ under $f^{-n(i)}$ inside $U(i)$.\\ This point $x$ is irregular and its signature $sign(x,0)$ with respect to the critical point $0$ contains a non-zero element $\beta\in A$ which is the set of classes $[b]$ of all binary sequences $b$ satisfying the following condition: for $b$ there exists a sequence $(\epsilon_1,\epsilon_2,...)$ of increasing positive integers, such that for any integer $j>0$, the entries $b_{\epsilon_j},b_{\epsilon_j+1},b_{\epsilon_j+2},...$ of $b$ can (but are not required to) be $1$ only if their index is of the form $1+n(t)+n(t+1)+...+n(m)$, where $t$ and $m$ are any integers satisfying $j\leq t\leq m$.\\ This element $\beta$ can not be described as $\alpha[b]$ for any binary sequence $b$, which implies that $\beta$ does not have a maximal element.
\end{prop}
\begin{proof}
Starting with $x_1=0$, we construct this irregular point $x=(x_1,x_2,...)$. For every $i=1,2,3,...$ both $f$ and $f(i)=f^{n(i)}$ are branched covering maps of degree $2$ from $U(i)$ onto $f(U(i))$ and $V(i)$, respectively. Hence, the map $f^{n(i)-1}$ is a univalent covering map from $f(U(i))$ onto $V(i)$. Both pre-images of $x_1$ in $U(i)$ under $f(i)=f^{n(i)}$ are obtained by taking the point $x_1=0\in V(i)$, and making $n(i)$ successive appropriate lifts by $f^{-1}$. The initial $n(i)-1$ lifts $x_2,x_3,...,x_{n(i)}$ of $x_1$ by $f^{-1}$ coincide for both of these pre-images of $x_1$ in $U(i)$. The two choices in the lift by $f^{-1}$ which distinguished between these two pre-images of $x_1$ in $U(i)$ appear at the last, $n(i)^{th}$, step. Since $V(j)$ is contained in $V(i)$ for all $j\ge i$, we see that the initial $n(i)-1$ choices of lifts by $f^{-1}$ which we make to lift $x_1=0\in V(j)$ to any one of its two pre-images in $U(j)$ under $f(j)$ coincide with $x_2,x_3,...,x_{n(i)}$.
\\ \\
Thus, we have constructed the unique point $x=(x_1,x_2,...)$ in the plaque inverse limit $S_{\infty}$ of $f:\mathbb{C}\rightarrow \mathbb{C}$, such that $x_1=0$ and each $x_{n(i)+1}$ is a pre-image of $x_1$ under $f^{-n(i)}$ inside $U(i)$. For every open, simply connected neighborhood $W_1$ of $0$ there exists some $j$, such that $V(j)$ is contained in $W_1$. Therefore, the lift $W_{n(j)+1}$ of $W_1$ along the pre-image $x_{n(j)+1}$ of $x_1\in W_1$ contains $U(j)$ and $U(j)$ contains $V(j+1)$. Hence, the lift $W_{n(j+1)+n(j)+1}$ of $W_1$ contains $U(j+1)$, which contains $V(j+2)$. So, $W_{n(j+2)+n(j+1)+n(j)+1}$ contains $U(j+2)$, which contains $V(j+3)$. By that logic, each lift $W_{n(j+m)+n(j+m-1)+...+n(j+1)+n(j)+1}$ will contain $V(j+m+1)$, for all $m$. Since all $V(1),V(2),...$ contain $0$, we get that the lift $W$ of $W_1$ to $S_{\infty}$ along $x$ will contain the critical point $0$ in infinite number of its levels. Thus, $x$ is an irregular point.
\\ \\
Now, let us perform some computations of the signature $sign(x,0)$ of $x$ with respect to the critical point $0$. For every positive integer $i$, let the neighborhood $\widetilde{V(i)}$ of $x$ in $S_{\infty}$ be obtained by lifting $V(i)$ along $x$. For every positive integer $j$, let $b'(j)$ be the binary sequence, which has $1$ in the places $1,n(j)+1,n(j+1)+n(j)+1,n(j+2)+n(j+1)+n(j)+1,...$ and $0$ in all other places. Now take the binary sequence: $$b(j)=\bigvee\limits_{t=j}^{\infty}b'(t).$$
In other words, $b(j)$ has $1$ in its $r^{th}$ place, where $r=1,2,...$, if and only if some $b'(t)$, with $t\ge j$, has $1$ in the $r^{th}$ place. It is clear from the arguments, made above to demonstrate that $x$ is irregular, that for $W_1=V(j)$ and $W=\widetilde{V(j)}$, we get $Ind(\widetilde{V(j)},0)\ge b(j)$. Our arguments do not establish an equality, since, for all $k$, $U(k)$ contains $V(k+1)$, but, in general, is not equal to $V(k+1)$. Hence, while $n(k+1)-1$ lifts of $V(k+1)$ along $x$ can not contain the critical point $0$, $n(k+1)-1$ lifts of $U(k)$ along $x$ may contain $0$. So, $$sign(x,0)\ge \bigcap\limits_{t=1}^{\infty}\alpha[b(t)].$$
\\
Note, that due to Corollary \ref{cor.finite}, the element $\beta=\bigcap\limits_{t=1}^{\infty}\alpha[b(t)]$ of $A$ can not be described as $\alpha[b]$ for any binary sequence $b$. This follows from the fact that the equivalence classes $[b(t)]$, as $t$ goes to infinity, can not stabilize, as required by Corollary \ref{cor.finite}, in order to have such a maximal element.
\end{proof}
To clarify the last point, that the equivalence classes $[b(t)]$, as $t$ goes to infinity, can not stabilize, and to understand $\beta$ a little better, let us consider a concrete example. Let $a=-1.401...$ be the Feigenbaum parameter and consider the map $f(z)=z^2+a$. It is known from the literature (see, for example, \cite{Buff}, page 114 in \cite{McMullen}, and page 34 in \cite{Jiang}), that $f(z)$ is infinitely renormalizable with $n(j)=2^j$ and has \textit{a priori} bounds.
\\ \\
First, note that in this case $b'(j)$ will have $1$ in the places $1,1+2^j,1+2^j+2^{j+1},1+2^j+2^{j+1}+2^{j+2},...$ and $0$ elsewhere. Next, $b(j)$ will have $1$ in the first place and in all the places of form $1+2^t+2^{t+1}+...+2^m$, where $j\le t\le m$. Hence, $\beta$ is the set of classes $[b]$ of all binary sequences $b$ which satisfy the following condition: for $b$ there exists a sequence $(\epsilon_1,\epsilon_2,...)$ of increasing positive integers, such that for any integer $j>0$, the entries $b_{\epsilon_j},b_{\epsilon_j+1},b_{\epsilon_j+2},...$ of $b$ can (but are not required to) be $1$ only if their index is of the form $1+2^t+2^{t+1}+...+2^m$, where $t$ and $m$ are any integers which satisfy $j\leq t\leq m$. For example, $\beta$ will contain an element $$[1,0,1,0,1,0,0,0,1,0,0,0,0,0,0,0,1,0...],$$ which has $1$ at the first and $(2^j+1)^{th}$ places, for $j=1,2,...$. This becomes clear if we take $\epsilon_j=1+2^j$. However, if we take $\epsilon_j=1+2^j+2^{j+1}$, we get that $\beta$ also contains a bigger element $$[1,0,1,0,1,0,1,0,1,0,0,0,1,0,0,0,1,0,0,0,0,0,0,0,1,0,...],$$
which has $1$ at the first, $(2^j+1)^{th}$, and $(2^{j+1}+2^j+1)^{th}$ places, for $j=1,2,...$. Similarly, increasing the rate of growth of the sequence $(\epsilon_1,\epsilon_2,...)$, we find bigger and bigger elements in $\beta$. Indeed, we can take $\epsilon_j=2^{h(j)}+2^{h(j)-1}+...+2^j+1$, where $h(j)$ is an arbitrarily ``big" integer function.
\\ \\
Now we proceed to another non-cyclical case in which irregular points are always present. For the definition of a Siegel disk see page 134 in \cite{Bea}, page 55 in \cite{CG}, and page 126 in \cite{Milnor}. For the definition of a Herman ring see page 160 in \cite{Bea}, page 74 in \cite{CG},
and page 161 in \cite{Milnor}. It is well known that the boundary components of the Siegel disks and the Herman rings are contained in the closure of the forward orbits of the critical points of $f$ (see, for example, Corollary 14.4 and Corollary 15.7 in \cite{Milnor}).
\\ \\
Lyubich and Minsky state and prove (on page 50 of \cite{LM}) the Shrinking Lemma. On page 8 of \cite{LM} the Shrinking Lemma is used to show, that every point, belonging to the invariant lift of the boundary of a Siegel disk or a Herman ring, is irregular. Due to its importance, we briefly reproduce this material (from page 8) in Lemma \ref{SiegelHerman} and its proof below. Then, we use Lemma \ref{SiegelHerman} to study the sets $\Omega(c)$ and $\Gamma(c)$ for certain cases of Siegel disks.
\\ \\
Let $W$ be an open Siegel disk or an open Herman ring in $S_0$, where $S_0$ is a complex plane or a Riemann sphere. Let $\partial W$ be the boundary of $W$. Let $\widehat{\partial W}$ be the invariant lift of $\partial W$ to the plaque inverse limit $S_{\infty}$. The following lemma and its proof appear, as a part of a discussion, on page 8 of \cite{LM}:
\begin{lem} \label{SiegelHerman} Any point $x\in \widehat{\partial W}$ is irregular.
\end{lem}
\begin{proof} Let $x=(x_1,x_2,...)$ be a point in $\widehat{\partial W}$. Note that any neighborhood of $x_1$ can not be entirely contained inside any rotation domain $-$ a Siegel disk or a Herman ring. Suppose that $x_1$ has a neighborhood $U_1$, such that for some positive integer $n$, all the lifts $U_n,U_{n+1},...$ of $U_1$ along $x$ do not contain any critical points of $f$. By the Shrinking Lemma, for any neighborhood $V_{n}$ of $x_{n}$, such that its closure $\overline{V_{n}}$ is contained in $U_{n}$, the diameters (in the spherical metric) of the lifts $V_{n+1},V_{n+2},...$ of $V_{n}$ along $x$ must converge to zero. But the restriction $f_W$ of $f$ to $W$ is a univalent map conjugated to a rotation. Hence, the diameters of the lifts by $f_W^{-1}$ of the intersection $V_{n}\cap W$, which are less than or equal to the corresponding diameters of $V_{n+1},V_{n+2},...$, can not shrink. This constitutes a contradiction. Thus, any point $x\in \widehat{\partial W}$ must be irregular.
\end{proof}
Now we investigate the case, in which $W$ is a Siegel disk and its boundary $\partial W$ is a Jordan curve. Several cases in which this occurs and $\partial W$ contains a critical point, are listed and studied in \cite{PZ}, \cite{Zhang}, and \cite{CR}. One such case is presented in the Main Theorem of \cite{Zhang}. It asserts that if $f$ is a rational map of a degree $\ge 2$ and has a Siegel disk $W$, with its rotation number $\theta$ being an irrational number of a bounded type, then $\partial W$ is a quasi-circle, which contains a critical point of $f$. A case, in which a boundary component $\partial W$ of a Siegel disk $W$ is a Jordan curve, which does not contain any critical points, is presented in \cite{Roesch}.
\\ \\
The map $f:W\rightarrow W$ is conjugated (by a conformal homeomorphism) to a rotation of the unit disk by the angle $2\pi\theta$. The Caratheodory Extension Theorem permits us to extend this conjugation of $f:W\rightarrow W$ to conjugation of $f:\overline{W}\rightarrow \overline{W}$ (by a
homeomorphism) to a rotation of the closed unit disk by the angle $2\pi\theta$. Thus, $f:\partial W\rightarrow \partial W$ is conjugated (by a homeomorphism) to a rotation of the unit circle $S^{1}$ by the angle $2\pi\theta$.
\begin{lem} If the boundary $\partial W$ of a Siegel disk $W$ is a Jordan curve then there exists a critical point $c$ of $f$ such that $\partial W$ is contained in $\Gamma(c)$.
\end{lem}
\begin{proof}
Since every point of $\widehat{\partial W}$ is irregular, there exists some $x\in \widehat{\partial W}$ and some critical point $c$, such that $sign(x,c)$ is not trivial. Let $[a]$ be a non-trivial class of binary sequences, which is contained in $sign(x,c)$ and let a binary sequence $a=(a_1,a_2,...)$ be a representative of this class. Then, for any neighborhood $U_1$ of $x_1$ in $S_0$, almost all the lifts $U_{a_1},U_{a_2},...$ of $U_1$ along $x$ will contain $c$.
\\ \\
For any point $y\in \widehat{\partial W}$ and any neighborhood $V_1$ of $y_1$ in $S_0$, since $f$, restricted to $\partial W$, is conjugated to a rotation of a circle by an irrational angle, there exists some number $t$, such that the lift $V_t=f^{-t}(V_1)$ of $V_1$ along $y$ will contain $x_1$. Hence, almost all lifts $V_{a_1+t},V_{a_2+t},...$ of $V_1$ along $y$ will contain $c$. Indeed, lifting $V_t$ along $(y_t,y_{t+1},...)$ is the same as lifting $V_t$ along $x$, since $f$ is bijective on $\partial W$. Therefore, for any neighborhood $V$ of $y$ in $\widehat{\partial W}$, an infinite number of levels of $V$ contain $c$. So, $sign(x,c)$ is not trivial. Thus, $\Gamma(c)$ contains the entire $\partial W$.
\end{proof}
For the case in which $\partial W$ is a Jordan curve which contains a critical point $c$, we can explicitly compute the signature $sign(x,c)$ for any point $x\in \widehat{\partial W}$, as follows:\\ \\
Let $\varphi:\partial W\rightarrow S^1$ be the unique homeomorphism, such that $\varphi(c)=+1$ and $\varphi\circ f\circ\varphi^{-1}=rot_{\theta}$, where $rot_{\theta}$ is the clockwise rotation of the $S^1$ by the angle $2\pi\theta$. Since $\theta$ is irrational, such a homomorphism $\varphi$ is unique, since it must take point $f^t(c)$ to $e^{-2t\pi\theta\sqrt{-1}}$ for all $t=1,2,...$. But, the forward orbits of $c$ are dense in $\partial W$. So, by continuity, $\varphi$ is uniquely determined at every point of $\partial W$. We use $\sqrt{-1}$, and not $i$, to denote the imaginary numbers, to avoid confusion with the index $i$.
\\ \\
Let $r_1=e^{2\pi\tau_1\sqrt{-1}}\in S^{1}$ be $\varphi(x_1)$. Let $(U(1),U(2),...)$ be a sequence of neighborhoods of $x$, converging to $x$ and for $i=1,2,...$, let $V(i)=\varphi(U(i)_1\cap \partial W)$. Clearly, $(V(1),V(2),...)$ is a sequence of neighborhoods of $r_1$ in $S^1$ which converges to $r_1$. The projection $U(i)_j\in S_j$ of $U(i)\in S_\infty$ contains $c$ if and only if the set $V(i)$, after a counterclockwise rotation by $2\pi\theta\cdot j$, contains $+1$. Let $(\epsilon(1),\epsilon(2),...)$ be any sequence of decreasing positive real numbers, which converges to $0$. Since the signature does not depend on a particular choice of a sequence of neighborhoods of a point, which converges to that point, we can assume that $(U(1),U(2),...)$ was selected in such a way that each $V(i)$ is the open arc of angle $2\pi\epsilon(i)$ in $S^1$, located between the points $e^{2\pi(\tau_1-\epsilon(i))\sqrt{-1}}$ and $e^{2\pi(\tau_1+\epsilon(i))\sqrt{-1}}$ of $S^1$.
\\ \\
For all $i=1,2,...$, there is a binary sequence $b(i)$ in the equivalence class $ind(U(i),c)$, which contains $1$ in its $j^{th}$ place if and only if the arc $V(i)$, after the counterclockwise rotation by the angle $2\pi\theta\cdot j$, contains the point $+1$. This is equivalent to requiring that for some integer $m$, the real number $\tau_1+\theta\cdot j$ satisfies $m-\epsilon(i)<\tau_1+\theta\cdot j<m+\epsilon(i)$. Thus, we obtain:
\begin{lem} \label{noncycleirreg} Let $W$ be a Siegel disk with a rotation number $\theta$, such that its boundary $\partial W$ is a Jordan curve, which contains a critical point $c$. Let $x=(x_1,x_2,...)$ be any point of the invariant lift $\widehat{\partial W}$ of $\partial W$ to the P.I.L. Denote by $\tau_1$ the angle of $\varphi(x_1)\in S^1$ (measured, going counterclockwise, from $+1$) divided by $2\pi$. The signature $sign(x,c)$ is the set of all the classes $[a]$ of binary sequences $a$, such that for $a$ there exists a sequence $(\delta_1,\delta_2,...)$ of decreasing positive real numbers, converging to $0$, so that for any positive integer $j$, the entries $a_j,a_{j+1},a_{j+2},...$ of $a$, can be $1$ only if $\tau_1+\theta\cdot j$ is within $\delta_j$ distance from some integer $m_j$.
\end{lem}
The signatures, described in Lemma \ref{noncycleirreg}, do not have a maximal element $-$ there does not exist a binary sequence $a$, such that $sign(x,c)=\alpha([a])$. This follows from the fact that the equivalence classes $ind(U(i),c)$, as $i$ goes to infinity, can not stabilize, as required by Corollary \ref{cor.finite}, in order to have such a maximal element.
\\ \\
Note that in the case described in Lemma \ref{noncycleirreg}, for a critical point $c\in \partial W$, we have $\partial W=\omega(c)=\Gamma(c)$ and $c\in \Omega(c)$.
\\ \\
One sees from Lemma \ref{noncycleirreg}, that the signature $sign(x,c)$, in the case when $x_1=c$, is closely related to the rational approximations of $\theta$. Indeed, since in this case $\tau_1=0$, the requirement of the lemma becomes $$\frac{m_j}{j}-\frac{\delta_j}{j}<\theta<\frac{m_j}{j}+\frac{\delta_j}{j}$$ for some integers $m_j$. Thus, we can take our $j$ from any sequence $(j_1,j_2,...)$ of increasing denominators of the rational approximations $\frac{m_{j_i}}{j_i}$ of $\theta$, satisfying that each $i^{th}$ approximation is within the $\frac{\delta_{j_i}}{j_i}$ distance to $\theta$. Here, $(\delta_1,\delta_2,...)$ is any sequence  of decreasing positive real numbers, converging to $0$. We refer to Chapter XI in \cite{HW} for further reading on such approximation of irrationals by rationals.
\\ \\
For example, Theorem 193 in \cite{HW} asserts that for any irrational number $\theta$ there are infinitely many rational solutions $\frac{m_{j_1}}{j_1},\frac{m_{j_2}}{j_2},...$, with $j_1<j_2<...$, to the equation $|\theta-\frac{m}{j}|<\frac{1}{\sqrt{5}j^2}$. Thus, if we define $\delta_j=\frac{1}{\sqrt{5}j}$, and construct a binary sequence $a$ by placing $1$ in all the places $j_i$, and $0$ in all the other places, then the equivalence class $[a]$ will be non-trivial and will be contained in $sign(x,c)$. By taking more and more ``slowly" converging sequences $(\delta_1,\delta_2,...)$, we can find bigger and bigger elements in $sign(x,c)$.


\begin{thebibliography}{1}

\bibitem{Bea} Alan F. Beardon.  Iteration of Rational Functions.
Springer-Verlag, 1991.

\bibitem{Berg} Walter Bergweiler.  On the number of critical points in parabolic
basins. Ergodic Theory and Dynamical Systems 22 (2002),  pp. 655-669.

\bibitem{Buff} Xavier Buff. Geometry of the Feigenbaum map. Conformal Geometry and Dynamics. An Electronic Journal of the American Mathematical Society 3 (August 12, 1999), pp. 79-101.

\bibitem{CK} Carlos Cabrera and Tomoki Kawahira.  On the natural extension of dynamics with a Siegel or a Cremer point. Journal of Difference Equations and Applications 19:5 (2013), pp. 701-711.

\bibitem{CK2} Carlos Cabrera and Tomoki Kawahira. Topology of the regular part for infinite renormalizable quadratic polynomials. Fundamenta Mathematicae 208 (2010), pp. 35-56.

\bibitem{CG} Lennart Carleson and Theodore W. Gamelin.  Complex Dynamics. Universitext / Universitext: Tracts in Mathematics (1993).

\bibitem{CR} Arnaud Cheritat and Pascale Roesch.  Herman's condition and Siegel
disks of polynomials. arXiv 1111.4629 [math.DS], November 20, 2011.

\bibitem{Danzig} David van Dantzig. Ueber topologisch homogene Kontinua. Fundamenta Mathematicae 15 (1930), pp. 102-125.

\bibitem{HW} Godfrey H. Hardy and Edward M. Wright. Revised by Roger Heath-Brown, Joseph Silverman and Andrew Wiles. An introduction to the Theory of Numbers (6th edition). Oxford science publications (2008).

\bibitem{Jiang} Ynping Jiang. The renormalization method and quadratic like maps.
arXiv math/9511208 [math.DS], November 02, 1995.

\bibitem{Jiang2} Ynping Jiang. Renormalization and Geometry in One-Dimensional and Complex Dynamics. Advanced Series in Nonlinear Dynamics 10 (1996).

\bibitem{Lyubich} Mikhail Lyubich.  The dynamics of rational transforms: the topological picture. Russian Mathematical Surveys 41:4 (1986),  pp. 43-118.

\bibitem{Lyubich2} Mikhail Lyubich.  On the Lebesgue measure of the Julia set of a quadratic polynomial.
arXiv math/9201285 [math.DS], May 28, 1991.

\bibitem{Lyubich3} Mikhail Lyubich. Renormalization ideas in conformal dynamics. Cambridge Seminar ``Current Developments in Mathematics", International Press (1995), pp. 155-184.

\bibitem{LM} Mikhail Lyubich and Yair Minsky.  Laminations in holomorphic
dynamics. Journal of Differential Geometry 47 (1997),  pp. 17-94.

\bibitem{McMullen} Curtis T McMullen. Complex Dynamics and Renormalization. Annals of Mathematical Studies 135. Princeton University Press (1994).

\bibitem{Milnor} John W. Milnor.  Dynamics in One Complex Variable. Princeton University Press (2006).

\bibitem{PZ} Carsten L. Petersen and Saeed Zakeri. On the Julia set of a typical quadratic polynomial with a Siegel disk. Annals of Mathematics 159
(2004), pp. 1-52.

\bibitem{Roesch} Pascale Roesch. Some rational maps whose Julia sets are not locally connected. Conformal Geometry and Dynamics. An Electronic Journal of the American Mathematical Society 10 (July 6, 2006), pp. 125-135.

\bibitem{Su} Meiyu Su. Measured solenoildal Riemann surfaces and holomorphic dynamics. Journal of Differential Geometry 45 (1997), pp. 170-195.

\bibitem{Sullivan} Dennis Sullivan. Bounds, quadratic differential, and renormalization conjectures. Mathematics into Twenty-First Century: 1988 Centennial Symposium, August 8-12, American Mathematical Society Centinnial Publications (1992), pp. 417-466.

\bibitem{Vietoris} Leopold Vietoris. \"{U}ber den h\"{o}heren Zusammenhang kompakter R\"{a}ume und eine Klasse von zusammenhangstreuen Abbildungen. Mathematische Annalen 97 (1927), pp. 454-472.

\bibitem{Zak} Saeed Zakeri.  On critical points of proper holomorphic maps on
the unit disk. Bulletin of the London Mathematical Society 30 (1998), pp. 62-66.

\bibitem{Zhang} Gaofei Zhang. All bounded type Siegel disks of rational maps are quasi-disks.
Inventiones mathematicae 185:2 (2011), pp 421-466.

\end{thebibliography}
 \end{document}